\documentclass[11pt]{article}

\usepackage{amsmath,amssymb,amsthm,graphicx,fixmath,tikz}

\def\Re{\mathbb R}

\providecommand{\remove}[1]{}

\theoremstyle{plain}
\newtheorem{theorem}{Theorem}[section]
\newtheorem{lemma}[theorem]{Lemma}

\newtheorem{corollary}[theorem]{Corollary}

\theoremstyle{definition}
\newtheorem{definition}[theorem]{Definition}
\theoremstyle{remark}
\newtheorem{remark}[theorem]{Remark}

\newcommand{\cardin}[1]{\lvert {#1} \rvert}

\begin{document}

\title{On Totally Positive Matrices and Geometric Incidences}

\author{%
 Miriam Farber\thanks{Mathematics department, Technion – Israel Institute of Technology,
Haifa, IL-32000, Israel. {\tt miriamf@tx.technion.ac.il}} \and
\and Saurabh Ray\thanks{Mathematics department, Ben-Gurion University,
Be'er Sheva 84105, Israel. {\tt saurabh@math.bgu.ac.il}} \and Shakhar
Smorodinsky\thanks{Mathematics department, Ben-Gurion University,
Be'er Sheva 84105, Israel. {\tt shakhar@math.bgu.ac.il}}}
\date{}
\maketitle

\begin{abstract}
A matrix is called totally positive if every minor of it is positive. Such matrices are well studied and have numerous applications in Mathematics and Computer Science. We study how many times the value of a minor can repeat in a totally positive matrix and show interesting connections with incidence problems in combinatorial geometry. We prove that the maximum possible number of repeated $d \times d$-minors in a $d \times n$ totally-positive matrix is $O(n^{d-\frac{d}{d+1}})$. For the case $d=2$ we also show that our bound is optimal. We consider some special families of totally postive matrices to show non-trivial lower bounds on the number of repeated minors. In doing so, we arrive at a new interesting problem: How many unit-area and axis-parallel rectangles can be spanned by two points in a set of $n$ points in the plane? This problem seems to be interesting in its own right especially since it seem to have a flavor of additive combinatorics and relate to interesting incidence problems where considering only the topology of the curves involved is not enough. We prove an upper bound of $O(n^{\frac{4}{3}})$ and provide a lower bound of $n^{1+\frac{1}{O(\log\log n)}}$.
\end{abstract}

\section{Introduction and preliminaries}
\subsection{Incidences between points and curves}
Let $L$ be a set of $n$ distinct lines and let $P$ be a set of $m$ distinct points in the plane.
Let $I(P,L)$ denote the number of incidences between points in $P$ and lines in $L$ i.e.,
$$
I(P,L) = \cardin{ \{(p,l): p \in P, l \in L, p \in l  \} }
$$
and put $i(m,n) = max_{|P|=m,|L|=n}\{I(P,L)\}$.
Erd{\H o}s provided a construction showing the lower bound $i(m,n) = \Omega((mn)^{2/3} +m + n)$ and also conjectured that this is also the asymptotic upper-bound.
His conjecture was settled by Szemer{\' e}di and Trotter in their 1983 seminal paper \cite{ST83}:

\begin{theorem}[\cite{ST83}]\label{lines-inc}
$i(m,n)= O(m^{\frac{2}{3}}n^{\frac{2}{3}} + m + n)$.
\end{theorem}

The Szemer{\' e}di  and Trotter bound can be generalized and be extended in many ways. In this paper we use one of these extensions to the so-called families of pseudolines:

\begin{definition}
A family $\Gamma$ of $n$ simple Jordan curves in the plane is said to be a family of {\em pseudolines}
if every pair of curves in $\Gamma$ intersect at most once.
\end{definition}

The original proof of Szemer{\' e}di and Trotter is rather involved and contains a huge constant hidden in the $O$-notation. A considerably simpler proof was obtained by Clarkson et al. in \cite{Clarkson-90}. Their paper contains also an extension of the same upper bound for points and a family of pseudolines. Another surprisingly simple proof of the same bound was obtained by Sz{\'e}kely \cite{Szekely97}. For more on geometric incidences, we refer the reader to the survey of Pach and Sharir \cite{PachSharir}.

\subsection{Totally positive matrices}
 Let $A$ be an $m \times n$ matrix i.e., $A$ has $m$ rows and $n$ columns. For two sets of indices $I \subset \{1,\ldots,m\}$ and $J \subset \{1,\ldots,n\}$ with the same cardinality, the minor $\Delta_{I,J}$ is defined to be $\Delta_{I,J}= det( A_{I,J})$ where $A_{I,J}$ is the submatrix of $A$ corresponding to the row set $I$ and the column set $J$. An $m\times n$ matrix is called totally-positive, ($TP$ in short), if every minor of it is positive. We call a matrix $TP_d$ if all its $i\times i$ minors are positive for $1 \leq i \leq d$.

TP-matrices play an important role in many areas of mathematics and computer science, including  discrete mathematics, probability and stochastic processes, and representation theory \cite{Fallat, Gasca}.
 Besides being positive, the minors in TP-matrices have beautiful combinatorial interpretations
 (See, e.g.,\cite{Brenti,Lind}). $TP_2$ matrices also relate to the so-caled {\em Monge} and {\em anti-Monge matrices}. A matrix is called strict monge (anti-Monge) matrix if for every $2\textrm{-by-}2$ submatrix, the sum of the diagonal entries is smaller than ( bigger than ) the sum of the off-diagonal entries. A matrix $A$ is $TP_2$ if and only if the matrix obtained by taking the entrywise logarithm of $A$ is strict anti-Monge. Monge matrices and their generalizations were widely studied in mathematics and computer science, see, e.g., \cite{Katarina,Rudolf,Tiskin}.

 Here, we further investigate the properties of minors of TP-matrices. We are primarily interested in the following question: Given a $d \times n$ TP-matrix, what can be said about the number of equal minors of a given order in such a matrix.
 In \cite{TPequal}, the authors provide a sharp asymptotic bounds in the case of $1\times 1$ minors. Namely, they prove that the maximum possible number of equal entries in an $n\times n$ $TP$-matrix is $\Theta(n^{4/3})$. They also discuss the positioning of equal entries in such a matrix and obtain relations to Bruhat order of permutations.
  They also leave the following question as an open problem: what is the maximum possible number of equal $2\times 2$ minors in a $d \textrm{-by-} n$ $TP$-matrix? We now recall several definitions and results that obtained in \cite{TPequal} for the case $d=2$. Let $\alpha_A = \big\{\{i,j\} | det(A[{1,2}|{i,j}])=\alpha \big\}$, where $A$ is a $2\textrm{-by-}n$ $TP$ matrix.
  A graph $G$ of order $n$ is called $TP$ attainable if there exists a labeling of its vertices, a
$2\textrm{-by-}n$ $TP$ matrix $A$ and a positive number $\alpha$ such that $E(G)\subseteq \alpha_A$.
\begin{theorem}[\cite{TPequal}]\label{planar}
Let $G$ be an outerplanar graph. Then $G$ is $TP$ attainable graph.
\end{theorem}
As a corollary, they obtain the following:
\begin{corollary}
The maximal number of equal $2\textrm{-by-}2$ minors in a $2\textrm{-by-}n$ $TP$ matrix is at least $2n-3$.
\end{corollary}                                                                                                                                                    In this paper, we obtain a full solution for the case $d=2$. We find the maximal number of equal $2\textrm{-by-}2$ minors in a $2\textrm{-by-}n$ $TP$ matrix, and we also present a method to construct such matrices. We also provide various upper and lower bounds in several special cases for other values of $d$. Relations to problems from additive combinatorics are also presented. We use various types of geometric incidences results. Those include point-lines, point-pseudolines and point-hyperplanes incidences. Furthermore, we show a fascinating relation between the problem of bounding the number of equal $2\times 2$ minors in a special class of $n\times n$ $TP_2$-matrices (referred to as grid-matrices) to an interesting geometric problem of bounding the number of unit-area axis-parallel rectangles in a planar (non-uniform) $n\times n$ grid.

\subsection{Paper organization}
In the next section, we discuss the number of equal $2 \times 2$ minors in a $2\times n$ TP-matrix. We obtain an upper bound on this number, and show that this bound is asymptotically tight. In section 3, we present an upper bound on the number of $d \times d$ equal minors in a $d\times n$ TP-matrix. We also give a lower bound for the number of distinct $d \times d$ minors in such a matrix for the cases in which $d << n$. In section 4, we discuss a special class of $n \times n $ $ TP_2$ matrices and use point-pseudolines incidences in order to provide an upper bound on the number of equal $2 \times 2$ minors in such matrices. We also provide an upper bound and lower bounds on the number of unit-area and axis-parallel rectangles that a set of $n$ points in the plane can span. In section 5 we discuss relations between our problems and problems from additive combinatorics and present some open problems.

\section{Totally-Positive matrices $2\times n$}
We now prove sharp asymptotic bounds on the maximum possible number of $2 \times 2$ equal minors in a $2\times n$ TP-matrix.

\begin{theorem}\label{two_by_n}
Let $A$ be a $2\times n$ TP-matrix. Then the maximum number of equal $2\times 2$ minors in $A$ is $O(n^{4/3})$. This bound is asymptotically tight in the worst case.
\end{theorem}

\begin{proof}
Assume without loss of generality that the value of such a minor is $1$ (since all the $2\times 2$ minors are nonzero, we can scale the matrix to have this property). Let $P=\{(x_i,y_i)| i=1,\ldots,n\}$ denote the column vectors of $A$. That is $x_i=A_{1 i}$ and $y_i=A_{2 i}$. We think of the $p_i$'s ($p_i=(x_i,y_i)$) as points in the plane.
Abusing the notation, we denote by $(p_i,p_j)$ the $2\times 2$ matrix whose first column is $p_i$ and the second column is $p_j$.
We define a set $L=\{l_1,\ldots,l_n\}$ of $n$ distinct lines as follows; For each point $p \in P$, $p=(a,b)$, let $l_p = \{(x,y) \in \Re^2| ay-bx=1\}$. Note that since $a$ and $b$ are entries in a totally positive matrix, both of them are nonzero, and hence $l_p$ is indeed a line. In addition, notice that for a given pair of indices $1\leq i < j\leq n$, the equality $\det(p_i,p_j)=1$ holds if and only if the point $p_j$ is incident with the line $l_{p_i}$. Thus, the number of $2\times 2$ minors of $A$ that are equal to $1$ is bounded from above by the number of incidences $I(P,L)$.
Hence, by Theorem~\ref{lines-inc} this number is bounded by $O(n^\frac{4}{3})$.

Next, we prove that the above upper-bound is asymptotically tight. For any $n$, we construct a $2 \times n$ totally positive matrix with $\Omega(n^\frac{4}{3})$ equal $2 \times 2$  minors.
We use the lower bound construction for incidences between $n/2$ points and $n/2$ lines with $\Omega(n^\frac{4}{3})$ many incidences. Let $P$ denote the set of points in such a construction and let $L$ be the set of lines. It is easy to verify that one can transform such a construction (using affine and projective transformations) to a construction satisfying the following constraints:
\begin{enumerate}
  \item No two lines are parallel.
  \item The slope of each line is positive.
  \item The $y$-intercept of each line is positive.
  \item every two points are linearly independent (as vectors in $\Re^2$).
  \item For every line $\ell$ in the construction, the line $\ell'$ parallel to $\ell$ and passing through the origin does not pass through any point of $P$.
  \item All the points lie in the first quadrant.
\end{enumerate}
Let $P'$ be an additional set of $n/2$ points constructed as follows: Fix a line $l \in L$ and let $d$ be the distance of the origin from $l$. We define the point $p_l$ to be the point on a line parallel to $l$ passing through the origin and having distance $1/d$ from the origin, such that $p_l$ lies in the first quadrant (the existence of such a point is guaranteed from (2) and (3)). Note that from (1), for each pair of distinct lines $l_1,l_2 \in L$, we have that $p_{l_1}$ and  $p_{l_2}$ are linearly independent. In addition, from (3) and the definition of $p_l$, for any point $p \in P$ incident with the line $l$, we have $\det(p_l,p)=1$. Put $Q=P \cup P'$. Order the points in $Q$ by $q_1,q_2, \ldots, q_n$, according to the non-decreasing slopes $\theta_1 \leq \theta_2 \leq \ldots \leq \theta_n$ of the lines passing through the origin and the $q_i$'s. We claim that these slopes are, in fact, monotonically increasing. Namely, that for $i \neq j$ , $\theta_i \neq \theta_j$. To see this, consider the following three possible cases.\\

$\left\{
\begin{array}{ll}
 p_{i}, p_{j} \in P , & \hbox{from (4) we have $\theta_i \neq \theta_j$;} \\
 p_{i}, p_{j} \in P', & \hbox{from (1), and the way we constructed the points in P' we have $\theta_i \neq \theta_j$;} \\
p_{i} \in P, p_{j} \in P', & \hbox{from (5) we have $\theta_i \neq \theta_j$.}
                                                                               \end{array}
                                                                             \right.$
\\

\noindent Let $A$ be the 2-by-$n$ matrix whose $i^{th}$ column is $\left(
                                                                                                                                     \begin{array}{c}
                                                                                                                                       x_i \\
                                                                                                                                       y_i \\
                                                                                                                                     \end{array}
                                                                                                                                   \right)
$ where $q_i=(x_i,y_i)$. Note that from (6) and the construction of $A$, we have that $A$ is a TP matrix, and it has at least $I(P,L)$ $2 \times 2$ minors that are equal to 1. However, $I(P,L)=\Omega(n^\frac{4}{3})$. This completes the lower-bound construction.
\end{proof}

\section{Incidences between Points and Hyperplanes in $\Re^d$ and $TP_d$}

\subsection{Repeated Minors}
Here, we prove that the number of equal $d\times d$ minors in a $d\times n$ TP-matrix is $O(n^{d-\frac{d}{d+1}})$. Note that this a generalization of the first part of Theorem~\ref{two_by_n}, since for $d=2$, $n^{d-\frac{d}{d+1}}=n^{\frac{4}{3}}$.
We need the following theorem of Apfelbaum and Sharir \cite{AS07} regarding incidences between points and hyperplanes in $\Re^d$.

\begin{theorem}\cite{AS07}\label{hyperplanes-inc}
If the point-hyperplane incidence graph of $m$ points and $n$ hyperplanes in $\Re^d$ does not contain a $K_{r,s}$ as a subgraph
for some fixed $r,s$ then the number of point-hyperplane incidences is $O((mn)^{1-\frac{1}{d+1}}+m+n)$  where the constant of proportionality depends on $r$ and $s$.
\end{theorem}

Let $A = A_{d \times n}$ be a TP-matrix, and let $P=\{p_i\}_{i=1}^n$ be the column vectors of $A$. For each ordered tuple $I$ of $(d-1)$ indices $1\leq i_1 < i_2<\cdots < i_{d-1} \leq n$, let $\pi_{I}$ be the hyperplane defined by the following linear equation:
$$
\begin{vmatrix}
  A_{1i_1} & A_{1i_2} & . & . & . & A_{1i_{d-1}}& x_1 \\
 A_{2i_1} & A_{2i_2} & . & . & . & A_{2i_{d-1}}& x_2 \\
 . & . & . & . & . & .&.\\
  . & .& . & . & . & . & .\\
  . & . & . & . & . & . & . \\
 A_{di_1} & A_{di_2} & . & . & . & A_{di_{d-1}} & x_d \\
\end{vmatrix}=1.
$$
Note that, for a $(d-1)$-tuple of indices $I= (i_1 , i_2 , \cdots , i_{d-1})$, and a point $p_k \in P$ for which $ i_{d-1} < k$, the corresponding $d \times d$ sub-matrix whose column vectors are $p_{i_1},p_{i_2},\ldots,p_{i_{d-1}},p_{k}$ has determinant $1$ if and only if the point $p_{k}$ is incident with the plane $\pi_{I}$. We need the following lemma:
\begin{lemma}\label{hyperplanes-lemma}
Let $A = A_{d \times n}$ be a TP-matrix. Then, for two ordered tuples $I$, $J$ of $d-1$ indices such that $I \neq J$, the hyperplanes $\pi_{I}$ and $\pi_{J}$ are distinct.
\end{lemma}

\begin{proof}
Since $I \neq J$, there exists an index $l \in \{1,2,\ldots,n\}$ such that $l \in I$ and $l \notin J$. Obviously, $p_l \in Span\{p_i | i \in I \}$. Consider the set $\{p_l\} \cup \{p_j | j \in J \}$. This set is consisted of $d$ column vectors of $A$, and since $A$ is totally positive, any set of $d$ columns of $A$ must be linearly independent. Therefore $p_l \notin Span\{p_j | j \in J \}$, and hence the ${d \times d}$ matrix $B = \left(
 \begin{array}{ccccc}
p_{j_1} & p_{j_2} & \ldots & p_{j_{d-1}} & p_l \\
\end{array}
\right)$ (where $J= (j_1 , j_2 , \cdots , j_{d-1})$) has a nonzero determinant. Consider the matrix $B' = \left(
\begin{array}{ccccc}
p_{j_1} & p_{j_2} & \ldots & p_{j_{d-1}} & \frac{1}{\det B} p_l \\
\end{array}
\right)$.
Note that $\det B' = 1$, and therefore, from the definition of $B'$, we can conclude that $\frac{1}{\det B} p_l$ is incident with the plane $\pi_{J}$. On the other hand, since $p_l \in Span\{p_i | i \in I \}$, then $\frac{1}{\det B} p_l$ is not incident with the plane $\pi_{I}$, since the corresponding determinant is equal to zero. Thus, the point $\frac{1}{det B} p_l$ is incident with $\pi_{J}$ but not with $\pi_{I}$. Therefore, $\pi_{I}$ and $\pi_{J}$ are indeed distinct, as asserted.
\end{proof}

Next, using Lemma~\ref{hyperplanes-lemma}, we show the connection between Theorem~\ref{hyperplanes-inc} and bounding the number of times a given $d \times d$ minor can repeat in a $d \times n$ TP-matrix $A_{d \times n}$.
\begin{theorem}\label{upperbound}
Let $A = A_{d \times n}$ be a TP-matrix. Let $f(A)$ be the number of $d \times d$ sub-matrices of $A$ having determinant $1$. Then $f(A)= O(n^{d-\frac{d}{d+1}})$.
\end{theorem}

\begin{proof}
First, note that the number $f(A)$ is upper bounded by the number of incidences between the $n$ points $P=\{p_i\}_{i=1}^n$
and the set of hyperplanes $\Pi=\{\pi_{I} : I\subset [n] \cardin{I}=d-1\}$.
Hence, our goal is to bound the number of such incidences. From Lemma~\ref{hyperplanes-lemma}, $\Pi$ has ${n \choose d-1}= O(n^{d-1})$ distinct hyperplanes. We will show now that the intersection of any two hyperplanes in $\Pi$ can contain at most $d-1$ points. Let us take two ordered tuples $I$, $J$ such that $I \neq J$. Then by Lemma~\ref{hyperplanes-lemma} the hyperplanes $\pi_{I}$ and $\pi_{J}$ are distinct, and hence their intersection is an affine space of dimension at most $d-2$. Thus, if $P$ has $d$ points that are lying in the intersection of $\pi_{I}$ and $\pi_{J}$, there exists a hyperplane that passes through the origin and contains those $d$ points, a contradiction to the assumption that $A$ is totally positive. Hence, the incidence graph of $P$ and $\Pi$ does not contain a $K_{d,2}$ subgraph. We can now apply Theorem~\ref{hyperplanes-inc} with $n$ points and ${n \choose d-1}$-hyperplanes and obtain the asserted bound.
\end{proof}
For the special case of $d=3$, the upper bound obtained in Theorem~\ref{upperbound} is $O(n^\frac{9}{4})$. This upper bound is not tight, and in order to present a better upper bound, we use the following result of Apfelbaum and Sharir \cite{AS07}:
\begin{theorem}\label{hyperplanes_threedim}
Let $P\subset \Re^3$ be a set of $m$ points and let $\Pi$ be a set of $n$ planes in $\Re^3$, such that no three points of $P$ are collinear. Then the number of incidences is bounded by $I(P, \Pi) = O(m^\frac{3}{5}n^\frac{4}{5} + m + n)$.
\end{theorem}
Now we are ready to prove the following theorem:
\begin{theorem}
Let $A = A_{3 \times n}$ be a TP-matrix. Then $f(A)= O(n^\frac{11}{5})$.
\end{theorem}
\begin{proof}
Following the proof of Theorem~\ref{upperbound}, we have $|\Pi|=O(n^{2})$, and since $A$ is totally positive, no three points of $P$ are collinear. Hence, by Theorem~\ref{hyperplanes_threedim}, $I(P, \Pi) = O(n^\frac{3}{5}(n^2)^\frac{4}{5} + n + n^2)=O(n^\frac{11}{5})$.
\end{proof}
A natural question is whether this bound is tight. Since it is not known whether the bound from Theorem~\ref{hyperplanes_threedim} is tight, our question is also open.

\subsection{Distinct Minors}
Here we investigate the number of distinct $d\times d$ minors in a $d \times n$ TP- matrix.
 We show the following:
\begin{theorem}
Let $d << n$. The number of distinct $d\times d$ minors in a $d \times n$ TP- matrix is $\Omega(n)$
\end{theorem}
\begin{proof}
Let $A$ be a $d \times n$ TP- matrix, and let $T=\big\{\det(B)\ |\ B \textrm{ is a } d\times d \textrm{ submatrix of } A\big\}$. Let $I$ be a $d-1$-tuple of indices $I= (i_1 , i_2 , \cdots , i_{d-1}),  1\leq i_1 < i_2 < \cdots < i_{d-1}  \leq n$, and let $t \in T$. Denote by $\pi_{t,I}$ the hyperplane that is defined by the following linear equation:
$$
\begin{vmatrix}
  A_{1i_1} & A_{1i_2} & . & . & . & A_{1i_{d-1}}& x_1 \\
 A_{2i_1} & A_{2i_2} & . & . & . & A_{2i_{d-1}}& x_2 \\
 . & . & . & . & . & .&.\\
  . & .& . & . & . & . & .\\
  . & . & . & . & . & . & . \\
 A_{di_1} & A_{di_2} & . & . & . & A_{di_{d-1}} & x_d \\
\end{vmatrix}=t.$$

Note that for $t_1 \neq t_2 $, $ \pi_{t_1,I}$ and $\pi_{t_2,I} $ are different parallel hyperplanes. From Lemma~\ref{hyperplanes-lemma}, $ \pi_{t,I} \neq \pi_{t,J} $ for $I \neq J$. Finally, using almost an identical proof, we get that for $t_1 \neq t_2 $, $I \neq J$ we have $ \pi_{t_1,I} \neq \pi_{t_2,J} $. Thus, the set $\Pi_{T}=\big\{\pi_{t,I}\ |\ t \in T, I \textrm{ is a } (d-1)\textrm{-tuple of indices }\big\}$ is of size $|T|{n \choose d-1}=O(|T|n^{d-1})$. From the definition of $\Pi_{T}$, $I(P, \Pi_{T})$ is bounded from below by the number of $d\times d$ minors of $A$, which is ${n \choose d} \geq (\frac{n}{d})^d$ (since every minor defines an incidence). On the other hand, similarly to the proof of Theorem~\ref{upperbound}, one can obtain that $I(P, \Pi_{T})=O((n(n^{d-1}|T|))^{1-\frac{1}{d+1}})$. Therefore, $(\frac{n}{d})^d\leq M(n^{d}|T|)^{1-\frac{1}{d+1}}$ for some positive constant $M$, and hence for $d << n$, $ |T|=\Omega(n)$.
\end{proof}

\section{Construction of $TP_k$ matrices and unit-area rectangles in non-uniform grids}
\label{construction}
Consider an $n\times n$ $TP_2$ matrix $A$.
In \cite{TPequal}, it was proven that the number of equal entries in $A$ is $O(n^{4/3})$ and that this bound is asymptotically tight in the worst case. We are interested in bounding the number of equal $2 \times 2$ minors in $A$. Unfortunately, so far, we can only obtain the following rather weak upper bound:

\begin{theorem}
Let $A$ be an $n\times n$ $TP_2$ matrix. Then the number of equal $2\times 2$ minors in $A$ is $O(n^{\frac{10}{3}})$.
\end{theorem}

\begin{proof}
The proof is is an easy corollary of Theorem~\ref{two_by_n}. Fix two row indices $I=\{i,j\}$ and consider the sub-matrix $A_{I,[n]}$. Theorem~\ref{two_by_n} implies that this matrix contains at most $O(n^{4/3})$ equal $2\times 2$ minors.
Since there are ${n \choose 2} = O(n^2)$ pairs of row indices, we immediately get the upper bound $O(n^2\cdot n^{4/3}) = O(n^{\frac{10}{3}})$, as asserted.
\end{proof}

Let us introduce an interesting class of $n\times n$ matrices with positive entries and positive $k\times k$ minors. For $k=2$, those matrices are $TP_2$, and we investigate the number of equal $2\times 2$ minors in such matrices.

\begin{theorem}\label{kbykconstruction}
Let $ a_k > a_{k-1}> \cdots > a_1 > 0 $, $ 0< b_k <b_{k-1}< \cdots <b_1$ be two sets of $k$ reals, $k \geq 2$. Let $A$ be a $k\times k$ matrix in which $A_{ij}=(b_i + a_j)^{k-1}$. Then

$$
\det(A)=\prod_{i=0}^{k-1}\binom{k-1}{i}\prod_{1\leq t < l \leq k} (a_l-a_t) \cdot \prod_{1\leq w < u \leq k}(b_w-b_u).
$$
\end{theorem}

\begin{proof}
$\det(A)$ is a polynomial of degree $k(k-1)$ in $Z[a_1,a_2, \ldots, a_k,b_1,b_2, \ldots, b_k]$. Let $A^{\{l,t\}}$ be the matrix that is obtained from $A$ by subtracting the $t^{th}$ column from the $l^{th}$ column. Note that $\det(A)=\det(A^{\{l,t\}})$. In addition, for $1 \leq m \leq k$, ${A^{\{l,t\}}}_{ml}=(b_m + a_l)^{k-1}-(b_m + a_t)^{k-1}$. Using the fact that for a natural number $n$, the polynomial $x^n - y^n$ is divisible (over $Z[x,y]$) by $x-y$, we get that $(b_m + a_l)^{k-1}-(b_m + a_t)^{k-1}$ is divisible by $(b_m + a_l)-(b_m + a_t) = a_l-a_t$. Thus, each entry in the $l^{th}$ column of $A^{\{l,t\}}$ is divisible by $a_l-a_t$, and therefore $a_l-a_t | \det(A^{\{l,t\}})$. By repeating the process for each pair of columns in $A$, and using the fact that $Z[a_1,a_2, \ldots, a_k,b_1,b_2, \ldots, b_k]$ is a unique factorization domain, we get $\prod_{1\leq t < l \leq k} (a_l-a_t)|\det(A)$. Applying the process on pairs of rows instead pairs of columns leads to $\prod_{1\leq w < u \leq k}(b_w-b_u)|\det(A)$, and therefore
\begin{center}
 $\prod_{1\leq t < l \leq k} (a_l-a_t) \cdot \prod_{1\leq w < u \leq k}(b_w-b_u) | \det(A)$.
\end{center}
On the other hand, the degree of
\begin{center}
$\prod_{1\leq t < l \leq k} (a_l-a_t) \cdot \prod_{1\leq w < u \leq k}(b_w-b_u)$
\end{center}
 is $\binom{k}{2}+\binom{k}{2}=k(k-1)$, and therefore there exists a constant $C$ such that
\begin{center}
$\det(A)=C\prod_{1\leq t < l \leq k} (a_l-a_t) \cdot \prod_{1\leq w < u \leq k}(b_w-b_u)$.
\end{center}
 In order to find $C$, consider the coefficient of the monomial
 \begin{center}
 ${a_k}^{k-1}{a_{k-1}}^{k-2}\ldots a_2{b_1}^{k-1}{b_2}^{k-2}\ldots b_{k-1}$.
\end{center}
In
\begin{center}
$C\prod_{1\leq t < l \leq k} (a_l-a_t) \cdot \prod_{1\leq w < u \leq k}(b_w-b_u)$,
\end{center}
the coefficient is $C$. On the other hand, while calculating $\det(A)$, among all the possible generalized diagonals in $A$, this monomial appears only when we multiply all the entries on the main diagonal, and its coefficient is $\prod_{i=0}^{k-1}\binom{k-1}{i}$. Therefore, $C=\prod_{i=0}^{k-1}\binom{k-1}{i}$
\end{proof}

As a corollary, we obtain the following:

\begin{corollary}
Let $ a_n > a_{n-1}> \cdots > a_1 > 0 $, $ 0< b_n <b_{n-1}< \cdots <b_1$ be two sets of $n$ reals, $n \geq 2$. Let $k>1$ be a natural number, and let $A$ be an $n\times n$ matrix in which $A_{ij}=(b_i + a_j)^{k-1}$. Then any $k\times k$ minor in $A$ is positive.
\end{corollary}

Note that for $k=2$, using Theorem~\ref{kbykconstruction}, the construction yields an $n \times n$ $TP_2$ matrix.
We use that construction in order to show a lower bound on the number of equal $2 \times 2$ minors in an $n \times n$ $TP_2$ matrix.

Let $P \subset \Re^2$  be the set of $n^2$ points in the plane formed by the Cartesian product $A \times B$, where $A = \{a_1,\ldots,a_n\}$ and $B= \{b_1, \ldots, b_n\}$. Each axis-parallel rectangle with corners belonging to $P$ corresponds to a $2\times 2$ minor of the matrix and the value of that minor equals the area of the rectangle.

Indeed, fix a pair of row indices $i < j$ and a pair of column indices $k < l$ and consider the corresponding $2\times 2$ minor with value:

$$
(b_i+a_k)(b_j+a_l) - (b_i+a_l)(b_j+a_k) = (a_l-a_k)(b_i-b_j).
$$

Notice that this is exactly the area of an axis-parallel rectangle in the plane whose opposite corners are the points $(a_l,b_j)$ and $(a_k,b_i)$.

We use this relation in order to construct a matrix with many $2\times 2$ equal minors:
\begin{theorem}\label{lowerboundconstruction}
 Let $[n] = \{1,,\ldots,n\}$ and consider the $TP_2$ $n\times n$ matrix $A$ constructed as above (i.e., $A_{i,j} = i+j$).
 Then there are $n^{2+\frac{1}{O(\log\log n)}}$ equal $2\times 2$ minors in $A$.
\end{theorem}
\begin{proof}
Let $G$ be the corresponding uniform grid $[n]^2$.
The number of times a rectangle of area $k\leq n/2$ appears in this grid is $\Omega(n^2 \times div(k))$ where $div(k)$ is the number of divisors of $k$. To see this, consider the $\frac{n^2}{4}$ choices of the bottom left corner of the rectangle, each one of the form $ (p,q) $, $1 \leq p \leq \frac{n}{2} $, $ 1 \leq q \leq \frac{n}{2}$. Once the bottom left corner is fixed, there are $div(k)$ ways of choosing the top right corner and once that is fixed, the other two corners are also fixed. Next, we use the following well known fact from number-theory ( see, e.g., Section 13.10 of \cite{Apostol}): There always exists a $k \leq n/2$ for which $div(k) = \Omega(n^{\frac{1}{\log\log n}})$ Such a $k$ gives the required bound.
\end{proof}
The above lower bound clearly serves as a lower bound on the number of equal $2 \times 2$ minors for $TP_2$ matrices.

Next, we turn our attention to provide a non-trivial upper bound on the number of equal $2\times 2$ minors in a $TP_2$ matrix constructed in such a way. That is, we provide an upper bound on the number of times an axis-parallel rectangle spanned by a (non-uniform) grid in the plane can have a given fixed area.
\begin{theorem}\label{uaub}
Let $A$ and $B$ be two sets of $n$ reals. Then the maximum number of times that a rectangle of unit area
is spanned by $A\times B$ is $O(n^{8/3})$.
\end{theorem}

\begin{proof}
 Let us fix the bottom left corner of a rectangle, that is let $a \in A$ and $b \in B$ be fixed.
 Consider all unit area rectangles $\mathcal{R}_{a,b}$ whose bottom-left corner is $(a,b)$. The loci of all points $(x,y)$ that are
 the upper-right corner of a rectangle in $\mathcal{R}_{a,b}$ lie on the hyperbola
 $$
 \gamma_{a,b} = \{(x,y): (x-a)(y-b)=1\}
 $$

 Let $\Gamma = \{\gamma_{a,b}| (a,b) \in A \times B\}$ be the set of $n^2$ hyperbolas. Note that the rectangle whose bottom-left corner is $(a,b)$ and whose top-right corner is $(c,d)$, has area $1$ if and only if the point $(c,d)$ is incident with the hyperbola $\gamma_{a,b}$. That is, the number of such unit area rectangles is bounded from above by the number of incidences between the $n^2$ points in $A\times B$ and $\Gamma$. These $n^2$ hyperbolas of $\Gamma$ are translates of each other and hence form a family of pseudolines.
 Since $O(n^2)$ pseudolines and $O(n^2)$ points may have at most $O( (n^2\cdot n^2)^{2/3})$ incidences, we get the asserted bound.
 \end{proof}

 \begin{remark}
 This upper bound does not apply to the number of repeated $2 \times 2$ minors of an arbitrary $TP_2$ matrix but rather to the class of matrices constructed as grid matrices. Nevertheless, we believe that the problem of providing sharp asymptotic bounds on the number of unit area rectangles in grid matrices is very
 interesting and non-trivial. Note also that in order to obtain our upper bound we reduce our problem to that of bounding the number of incidences between a (not necessarily uniform) grid with $n^2$ points and a family of $n^2$ translates of a hyperbola. The bound $O(n^{8/3})$ that we get follows from the fact that those hyperbolas form a family of pseudolines. For pseudolines, this bound can be matched with a lower bound by modifying the construction of Elekes \cite{Elekes-2001} for a lower bound on the number of incidences between points and lines (see also the survey \cite{PachSharir}). The grid in Elekes' construction is rectangular but it can be easily turned into a square grid by appropriately cutting the longer side into pieces and putting them together along the shorter side. The lines don't remain straight any more but we can bend them in such a way that they remain pseudolines. Figure \ref{fig:elekes} illustrates the idea.

   \begin{figure}
     \centering
     \includegraphics[scale=0.70]{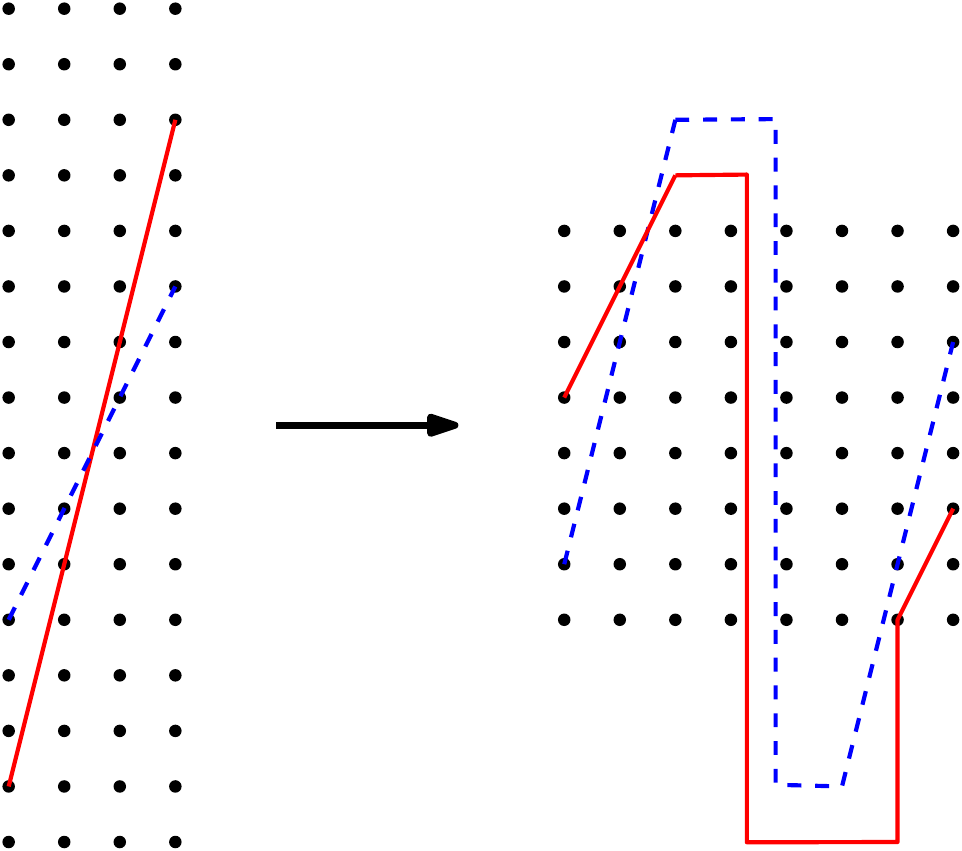}
     \caption{Turning the rectangular grid into a square grid in Elekes' construction}
     \label{fig:elekes}
   \end{figure}
 \end{remark}

In the proof of Theorem~\ref{uaub} we do not use the fact that the point set is a cartesian product of two sets of reals. Therefore, our proof technique provides the following more general theorem:
\begin{theorem}
Let $P$ be a set of $n$ points in the plane. Then $P$ determines at most $O(n^{4/3})$ unit-area and axis-parallel rectangles.
\end{theorem}

Notice also that the lower bound construction from Theorem~\ref{lowerboundconstruction} with $[\lfloor \sqrt n \rfloor ]^2$ in fact provides the following lower bound for unit-area axis-parallel rectangles spanned by a set of $n$ points:
\begin{theorem}
The maximum possible number of unit-area and axis-parallel rectangles spanned by $n$ points in the plane is $n^{1+\frac{1}{O(\log\log\sqrt n)}}$
\end{theorem}

 \section{Discussion and open problems}

The paper ~\cite{TPequal} studies the maximum possible number of repeated entries in an $n \times n$ $TP_2$-matrix
and asks what can be said about the maximum possible number of repeated $2 \times 2$ minors in a $d \times n$ $TP$-matrix. This problem is open already for $d=2$.
In this paper we initiate the study of repeated minors and provide several bounds. We completely settle the case $d=2$ and also provide non-trivial upper bounds on repeated $d \times d$ minors in a $d \times n$ TP-matrix. We show that higher dimensional geometric incidence results come into play.

We also study the problem of repeated $2 \times 2$ minors in a special class of  $n \times n$ TP-matrices.
This special class seems to be interesting in its own right due to its additive combinatorics flavor:
Consider the problem of repeated $2 \times 2$ minors in the special setting of Section~\ref{construction}.
Let us state the problem in a slightly different way to see that the problem has a flavor of additive combinatorics. For any two multisets $C$ and $D$, let us define $C - D$ to be the multiset $S$ where each element of $S$ is the difference of an element in $C$ and an element in $D$.
The multiplicity of any $s \in S$ is $m(s)= \sum_{c\in C, d \in D, c-d = s} m(c)\cdot m(d)$ where $m(x)$ is the multiplicity of $x$. Similarly define $C\cdot D$ to be the set $S'$ where each element in $S'$ is the product of an element in $C$ and an element in $D$ and the multiplicity of $s \in S'$ is $m(s) = \sum_{c\in C, d \in D, c\cdot d = s} m(c)\cdot m(d)$.
 For any multiset $C$ define $\mu(C)$ to be the maximum multiplicity of any element in $C$. We treat any set $C$ as a multiset where the multiplicity of each element is $1$.

Theorem~\ref{uaub} can then be rephrased as follows: for any sets $A$ and $B$ of $n$ reals each, $\mu((A-A)\cdot (B-B)) = O(n^{8/3})$. We believe that the correct bound is $O(n^{2+o(1)})$ but improving the $O(n^{8/3})$ seems difficult even for special cases. For instance, if we fix the set $B$ to be $\{1,\ldots, n\}$ then the problem boils down to showing that the sum of multiplicities of elements in $x \in A-A$ s.t. $1/x  \in \{1,\cdots,n\}$ is $O(n^{1+o(1)})$. Proving this special case may involve interesting number theoretic and graph theoretic techniques.

The special type of $TP_2$-matrices introduced in Section~\ref{construction} lead us to the following open problem which is of independent interest. Provide sharp asymptotic bound on the maximum possible number of unit-area and axis-parallel rectangle that are spanned by a pair of points in an $n$ point set in the plane.
To the best of our knowledge this fascinating problem has not been studied before.
In Section~\ref{construction} we provide an upper bound of $O(n^{4/3})$ and a lower bound of $n^{1+\frac{1}{O(\log\log\sqrt n)}}$.
We note that both our upper bound and lower bounds are similar to the best known upper and lower bounds for, yet another, classical problem in discrete geometry, the so-called Erd{\H o}s repeated distances problem. It might be the case that there is a direct relation between these two functions.

\end{document}